\documentclass[11pt]{article}

\usepackage[margin = 1.0in]{geometry}   
\usepackage{geometry}  			
\usepackage{amssymb}
\usepackage{amsmath}
\usepackage{palatino}
\usepackage{graphicx}
\usepackage[english]{babel}
\usepackage{amscd}
\usepackage{amsthm}
\usepackage{color}
\usepackage{enumerate}
\usepackage{hyperref}
\usepackage{subcaption}
\usepackage{comment}

\newtheorem{theorem}{Theorem}

\newtheorem{corollary}{Corollary}
\newtheorem{lemma}{Lemma}
\newtheorem{proposition}{Proposition}

\newcommand{\Prb}{\mathbb{P}}
\newcommand{\Exp}{\mathbb{E}}

\title{Concentration inequalities for the sum in sampling without replacement: an approach via majorization}

\author{
Jianhang Ai\thanks{Faculty of Electrical Engineering, Czech Technical University, Karlovo N\'am\v{e}st\'i 13, 12135, Prague 2, Czech Republic, e-mail: ai.jianhang@fel.cvut.cz}
\and
Ond\v{r}ej Ku\v{z}elka\thanks{Faculty of Electrical Engineering, Czech Technical University, Karlovo N\'am\v{e}st\'i 13, 12135, Prague 2, Czech Republic, e-mail: ondrej.kuzelka@fel.cvut.cz}
\and
Christos Pelekis\thanks{Aristotle University of Thessaloniki, 
Department of Mathematics, 
541 24 Thessaloniki, Greece, e-mail: pelekis.chr@gmail.com, \, cpelekis@math.auth.gr }
}

\begin{document}

\maketitle

\begin{abstract}  
Let $P=(x_1,\ldots,x_n)$ be a population consisting of $n\ge 2$ real numbers whose sum is zero, and let $k <n$ be a   positive integer. We sample $k$ elements from $P$ without replacement and denote by $X_P$ the sum of the elements in our sample. In this article, using  ideas from the theory of majorization, we deduce non-asymptotic lower and upper bounds on the probability that $X_P$ exceeds its expected value.
\end{abstract}

\noindent
\emph{Keywords}: probability inequalities; sampling without replacement;  majorization; non-negative $k$-sums;
Schur convexity; hypergeometric distribution; mean absolute deviation

\noindent
\emph{MSC (2020)}: 62D05; 60E15; 26B25

\section{Introduction}
\subsection{Notation}
Let us begin with introducing some notation that will be fixed throughout the text. 
Given a positive integer $n$, we denote by $[n]$ the set $\{1,\ldots,n\}$. 
The cardinality of a finite set $F$ is denoted by $|F|$. 
Given two positive integers $n,k$  such that $1\le k < n$, we denote by $\binom{[n]}{k}$  the collection consisting of all subsets $F$ of $[n]$  that satisfy $|F|=k$. 

In this article, we shall be concerned with estimates on the probability that the sum in sampling without replacement deviates from its expected value. 
Our approach is based on the notion of \emph{majorization of vectors}, which we briefly recall here and refer the reader to~\cite{Marshall_Olkin_Arnold} for further details and references. 
If $\mathbf{x}= (x_1,\ldots,x_n)\in\mathbb{R}^n$ is a vector,  let $x_{[1]} \ge \cdots \ge x_{[n]}$ denote the coordinates of $\mathbf{x}$ in decreasing order. Given two vectors $\mathbf{x}= (x_1,\ldots,x_n)$ and $\mathbf{y}= (y_1,\ldots,y_n)$  in $\mathbb{R}^n$, we say that $\mathbf{y}$ \textit{dominates} $\mathbf{x}$ (or that $\mathbf{x}$ \textit{is majorized} 
by $\mathbf{y}$) if it holds 
\[
\sum_{j=1}^{\ell} x_{[j]} \le \sum_{j=1}^{\ell} y_{[j]}\, , \,  \text{ for all } \, \ell=1,\ldots,n-1 \, , \quad \text{ and } \quad \sum_{j=1}^{n} x_{[j]} = \sum_{j=1}^{n} y_{[j]} \, . 
\]
We write $\mathbf{x}\prec\mathbf{y}$ when the vector  $\mathbf{x}$  is majorized by the vector $\mathbf{y}$. 
A function $f:A\to \mathbb{R}$ defined on a set $A\subseteq\mathbb{R}^n$ is said to be 
\emph{Schur convex on $A$} if for every  $\mathbf{x},\mathbf{y}\in A$ for which  $\mathbf{x}\prec \mathbf{y}$ we have $f(\mathbf{x})\le f(\mathbf{y})$.

Given two random variables $X,Y$, we write $X\sim Y$ when $X$ and $Y$ have the same distribution.
Given $i\in [n-1]$, we denote by $\text{Hyp}(n,i,k)$ 
a hypergeometric random variable that counts the number of black marbles in a sample without replacement of size $k$ from an urn that contains $i$ black and $n-i$ white marbles. Hence, for $H_i\sim\text{Hyp}(n,i,k)$, we have $\Prb(H_i=m)=\frac{\binom{i}{m}\cdot \binom{n-i}{m}}{\binom{n}{k}}$, for $m\in \{0,1,\ldots,\min\{i,k\}\}$. 
 
\subsection{Related work and main results}
Let $P=(x_1,\ldots,x_n)$ be a population consisting of $n\ge 2$ real numbers such that $\sum_{i=1}^{n}x_i=0$, and let $k<n$ be a positive integer. We sample $k$ elements from $P$ without replacement and consider the sum of the elements in our sample. 
More concretely, choose a set $\mathbb{I}\in\binom{[n]}{k}$ uniformly at random and let
\[
X_P = \sum_{i\in\mathbb{I}} x_i \, .
\]
How much does $X_P$ fluctuate from its expected value? 
A bit more precisely, let $t\ge 0$ be fixed and observe that $\Exp(X_P)=0$. What is an upper and a lower bound on the tail probability $\Prb(X_P \ge t)$?

The problem is of fundamental interest, both in theory as well as in applications, and has attracted  quite a bit of attention, although not on a scale one would expect.  Probably the first systematic approach to this problem was reported by Hoeffding in his seminal paper~\cite{Hoeffding} who, by employing his well-known result on the convex-ordering-dominance on  the sum in sampling without replacement by the sum in sampling with replacement, established the following. 

\begin{theorem}[Hoeffding~\cite{Hoeffding}]
\label{thm_Hoeffding}
Let $P=(x_1,\ldots,x_n)$ be a population of size $n\ge 2$ such that $\sum_{i=1}^{n} x_i=0$. 
Let $X_P$ denote the sum of $k\in [n-1]$ elements that are sampled  without replacement from $P$. If $a = \min_{i \in [n]} x_i$ and $b = \max_{i\in [n]} x_i$, then 
\begin{equation*}
\Prb(X_P \ge t) \le e^{-\frac{2t^2}{k(b-a)^2}} \, , \, \text{ for } \, t > 0 \, ,
\end{equation*}  
\end{theorem}

Hoeffding's bound has been refined and/or reproved  in~\cite{Serfling, McDiarmid, Bardenet_Maillard}, possibly among  others, but the results on the topic appear to be scarce.  
Notice that the bound requires information on the \emph{mean} and the \emph{extreme values} (i.e., the parameters $a,b$) of the population. Let us also remark that  the bound can be improved further under additional information on the population \emph{variance} (see~\cite{Bardenet_Maillard}, or Theorem~\ref{thms} below).
Note also that Hoeffding's bound refers to the case $t>0$ in the aforementioned question, and that it becomes trivial  when $t=0$; a fact  that appears to hold true for the majority of the existing  results regarding upper bounds. In particular, we were unable to find non-asymptotic upper, or lower, bounds on the tail probability   $\Prb(X_P\ge 0)$ that are valid for all values of $n\ge 2$ and $k\in [n-1]$.   

Rather surprisingly, 
the literature dealing with lower bounds on the tail  probability  $\Prb(X_P\ge t)$ appears to be richer. While the literature dealing with the case $t>0$ seems to be non-existent, the case $t=0$ has attracted an immense attention due to its relevance to a notorious problem in combinatorics which has been open since the late $80$'s; \emph{the Manickham-Mikl\'os-Singhi (MMS) conjecture}.  
The MMS conjecture asserts that, when $n\ge 4k$, it holds $\Prb(X_P\ge 0)\ge \frac{k}{n}$. The bound is sharp, as can be seen by considering the population $P=(1,-\frac{1}{n-1},\ldots,-\frac{1}{n-1})$. We refer the reader to~\cite{AHS} or~\cite{Pokrovskiy} for the history of the MMS conjecture and for further details and references. 
In particular, the following result holds true, and appears to be the state of the art on the topic.  

\begin{theorem}[Pokrovskiy~\cite{Pokrovskiy}]\label{Pokrovskiy}
Let $P=(x_1,\ldots,x_n)$ be a population of size $n\ge 2$ such that $\sum_{i=1}^{n} x_i=0$. 
Let $X_P$ denote the sum of $k\in [n-1]$ elements that are sampled without replacement from $P$. 
Assume further that it holds $n\ge 10^{46}\cdot k$. Then
\[
\Prb(X_P\ge 0) \ge  \frac{k}{n} \, .
\] 
\end{theorem}

In this article, we  deduce bounds on the tail probability $\Prb(X_P\geq t)$, for $t\geq 0$. 
When $t=0$, we obtain both lower and upper bounds on  $\Prb(X_P\ge 0)$ that are valid for all values of $n\ge 2$ and $k\in [n-1]$. 
When $t>0$, we obtain an upper bound on $\Prb(X_P\geq t)$ which requires  information on the absolute deviation of the population. 
Our bounds regarding the first case read as follows.

\begin{theorem}\label{main_thm}
Let $P=(x_1,\ldots,x_n)$ be a population of size $n\ge 2$ such that $\sum_{i=1}^{n} x_i=0$ and $\sum_{i=1}^{n}|x_i|>0$. 
Let $X_P$ denote the sum of $k\in [n-1]$ elements that are sampled  without replacement from $P$. Then
\[
\Prb(X_P> 0) \,\ge\,\frac{e^{-1/3}}{8\sqrt{2\pi}} \cdot \frac{k}{n}\cdot \sqrt{\frac{n-k}{n\cdot k}}  \, .
\]
\end{theorem}

We obtain Theorem~\ref{main_thm} using ideas from the theory of Schur-convex functions. 
Clearly, the bound provided by Theorem~\ref{main_thm} also applies to the tail probability $\Prb(X_P\ge 0)$. 
Furthermore, due to symmetry, we deduce the following upper bound on the aforementioned tail.

\begin{corollary}\label{main_thm4}
With the same hypotheses as in Theorem~\ref{main_thm}, it holds  
\[
\Prb(X_P\ge 0) \,\le\, 1 -\frac{e^{-1/3}}{8\sqrt{2\pi}} \cdot \frac{k}{n}\cdot \sqrt{\frac{n-k}{n\cdot k}} \, .  
\]
\end{corollary}

As already mentioned, our bounds rely on ideas from the theory of majorization. 
In particular, a crucial idea in the proofs, which is an idea that renders the problem amenable to basic results from Schur-convexity, is to consider the problem under the additional assumption that the  population $P=(x_1,\ldots,x_n)$  satisfies $\sum_{i=1}^n |x_i|=2\alpha$, for some $\alpha >0$, and  deduce a bound for such a population.  It turns out that, for the case $t=0$, the resulting lower and upper bounds are \emph{independent} of the parameter $\alpha$. 
The same idea applies to the  case $t>0$, but the corresponding bounds do depend on $\alpha$. In particular, we obtain the following upper bound for this case. 

\begin{theorem}\label{main_thm2}
Let $P=(x_1,\ldots,x_n)$ be a population of size $n\ge 2$ such that $\sum_{i=1}^{n} x_i=0$ and $\sum_{i=1}^{n} |x_i|=2\alpha$, for some $\alpha>0$.  
Let $X_P$ denote the sum of $k\in [n-1]$ elements that are sampled  without replacement from $P$, and let $t\in (0,\alpha)$ be fixed. Then   
\[
\Prb(X_P\ge t)\, \le\, 1 - \min\left\{1\, ,\, \frac{t}{\alpha -t}\right\}\cdot\left(1 - \frac{2k(n-k)}{n(n-1)}\right) \, .
\]
\end{theorem}

Observe that Hoeffding's bound requires knowledge on the minimum and the maximum value of the population while our bound requires knowledge on the absolute deviation of the population; hence it may be seen as being complementary to Theorem~\ref{thm_Hoeffding}. 

Finally, let us remark that our approach also yields a lower bound on $\Prb(X_P\ge t)$, for $t>0$, which is applicable to moderate values of $t > 0$. Due to its limited range of applicability, we defer the statement, and the proof, of this lower bound to  Appendix~\ref{appendix1}.

\subsection{Organization}
The remaining part of our article is organized as follows. 
In Section~\ref{sec:1}, we prove Theorem~\ref{main_thm}. 
The proof is based on ideas from the theory of majorization, combined with lower estimates on the mean absolute deviation of a hypergeometric random variable. 
The proof of Theorem~\ref{main_thm2} is obtained in a similar manner, and is deferred to Section~\ref{sec:2}. Finally, in Section~\ref{sec:3}, we  present pictorial comparisons between the bound provided by Theorem~\ref{main_thm2} and certain existing refinements of Theorem~\ref{thm_Hoeffding}. 
\section{Proof of Theorem~\ref{main_thm}}
\label{sec:1}

This section is devoted to the proof of Theorem~\ref{main_thm}. 
Let us begin with some observations. 
Given a real-valued random variable $X$ such that $\Exp(X)=0$, let 
\[
X^+ = \max\{0\, ,\, X\} \quad \text{ and } \quad X^- = \max\{0\, ,\, -X\} 
\]
and observe that $X = X^+ - X^-$ and that $|X|=X^+ + X^-$\, ; hence it holds 

\begin{equation}\label{eqx1}
\Exp(X^+) = \Exp(X^-) \quad \text{ and } \quad \Exp(|X|) = 2\cdot \Exp(X^+) \, .
\end{equation}
The proof of Theorem~\ref{main_thm} is based upon the following observation. 

\begin{lemma}[Folklore]
\label{folklore}
Let $X$ be a real-valued random variable  such that $\Exp(X)=0$. Then 
\[
\Prb(X> 0) = \frac{\Exp(|X|)}{2\cdot \Exp(X\, |\, X> 0)} \, .
\]
\end{lemma}
\begin{proof}
Observe that it holds   
\[
\Exp(X^+) = \Exp(\max\{0\, ,\, X\}) = \Exp(X\, |\, X\ge 0)\cdot \Prb(X\ge 0) 
\]
and   
\[
\Exp(X|X\ge 0)\cdot \Prb(X\ge 0) = \Exp(X|X> 0)\cdot \Prb(X> 0) \, .
\]
The desired result follows from~\eqref{eqx1}.  
\end{proof}

Now let $P=(x_1,\ldots,x_n)$ be a population of size $n\ge 2$ which satisfies $\sum_{i=1}^{n}x_i =0$, and let $k\in [n-1]$ be fixed.   
Choose $\mathbb{I}\in\binom{[n]}{k}$ uniformly at random and let $X_P = \sum_{i\in\mathbb{I}}x_i$ be the sum of the elements in the sample. Note that $\Exp(X_P)=0$.
We aim to apply Lemma~\ref{folklore} to $X_P$. 
We do so by  imposing ``additional structure" on the population. More concretely, 
we view a population $P=(x_1,\ldots,x_n)$ as a vector in $\mathbb{R}^n$. 
Given $\alpha>0$, define the set 
\begin{equation}\label{B_a_set}
B_{\alpha}^{(n)} = \left\{P=(x_1,\ldots,x_n)\in\mathbb{R}^n \, : \, \sum_{i=1}^n x_i = 0 \, \,\, \& \,\, \, \sum_{i=1}^n |x_i| = 2\alpha \right\} \, .
\end{equation}
We will obtain a lower bound on $\Prb(X_P> 0)$, for $P\in B_{\alpha}^{(n)}$, which will turn out to be independent of $\alpha$. 
In fact, we will show that the bound provided by Theorem~\ref{main_thm} holds true for all $\alpha>0$ and all $P\in B_{\alpha}^{(n)}$. 
Towards this end, observe that the assumption  $\sum_{i=1}^{n}|x_i|>0$ implies that there exists  $\alpha>0$ such that $P\in B_{\alpha}^{(n)}$; hence it holds $X_P\le \alpha$ and therefore   we have 
\begin{equation}\label{eqx2}
\Exp(X_P \, |\, X_P > 0) \le \alpha \, ,\, \text{ for } \, P\in B_{\alpha}^{(n)} \, .
\end{equation}
\noindent
Since $\Exp(X_P)=0$, it follows from~\eqref{eqx2} and Lemma~\ref{folklore}  that 

\begin{equation}\label{eqx3}
\Prb(X_P > 0) = \frac{\Exp(|X_P|)}{2\cdot \Exp(X_P\, |\, X_P> 0)} \ge \frac{\Exp(|X_P|)}{2\cdot \alpha} \, ,\, \text{ for } \, P\in B_{\alpha}^{(n)} \, .
\end{equation}
\noindent
Hence it is enough to estimate  $\Exp(|X_P|)$, for $P\in B_{\alpha}^{(n)}$, from below. 
This will be accomplished via Schur convexity. 
Note that 
\[
\Exp(|X_P|) = \frac{1}{\binom{n}{k}} \sum_{F\in \binom{[n]}{k}} |\sum_{i\in F} x_i| \, .
\]
The next result guarantees the Schur-convexity $\Exp(|X_P|)$.

\begin{lemma}
\label{Schur1}
Let $f:\mathbb{R}^n\to \mathbb{R}$ be the function defined by 
\[
f(x_1,\ldots,x_n) = \frac{1}{\binom{n}{k}} \sum_{F\in \binom{[n]}{k}} |\sum_{i\in F} x_i| \, .
\]
Then $f$ is Schur-convex on $\mathbb{R}^n$. 
\end{lemma}
\begin{proof}
For every $F\in\binom{[n]}{k}$, let $\chi_F:\mathbb{R}^n\to\mathbb{R}$ denote the function defined via 
\[
\chi_F(x_1,\ldots,x_n) = |\sum_{i\in F} x_i|  
\]
and note that $\chi_F$ is convex. 
Hence $f$ is convex. Moreover, $f$ is clearly symmetric under permutations of the arguments. Since $f$ is convex and symmetric, it follows (see~\cite[p.~97]{Marshall_Olkin_Arnold}) that $f$ is Schur-convex, as desired. 
\end{proof}

It follows from Lemma~\ref{Schur1} that whenever $P\prec P^{\prime}$, then $\Exp(|X_P|) \le \Exp(|X_{P^{\prime}}|)$. 
Hence it is enough to find ``minimal" (with respect to majorization) populations. 
We focus on the set $B_{\alpha}^{(n)}$, defined in~\eqref{B_a_set}. 
Let $\alpha>0$ be fixed. For $i\in [n-1]$, let $P_{i,\alpha}\in\mathbb{R}^n$ denote the vector 
\begin{equation}\label{p_vector}
  P_{i,\alpha}=\left(\frac{\alpha}{i},\ldots,\frac{\alpha}{i}, -\frac{\alpha}{n-i}, \ldots,-\frac{\alpha}{n-i}\right) \, ,
\end{equation}
whose first $i$ coordinates are equal to $\frac{\alpha}{i}$ and the remaining $n-i$ coordinates are equal to $-\frac{\alpha}{n-i}$.
Observe that $P_{i,\alpha}\in B_{\alpha}^{(n)}$.

\begin{lemma}\label{minimal}
Fix $\alpha>0$. Then for every $P\in B_{\alpha}^{(n)}$ there exists 
$i\in [n-1]$ such that $P_{i,\alpha} \prec P$.  
\end{lemma}
\begin{proof}
Let the coordinates of $P$ be written in decreasing order 
$x_1\ge x_2\ge\cdots\ge x_n$. 
Let $i$ be the largest integer such that $x_i\ge 0$. 
We claim that $x_1\ge \frac{\alpha}{i}$. Indeed, if this is not the case then $x_1<\frac{\alpha}{i}$ and thus 
$\sum_{j=1}^{i}x_j < \alpha$, which contradicts the fact that $\sum_{j=1}^{i}x_j=\alpha$. 
Similarly, it holds $x_1+x_2 \ge 2\frac{\alpha}{i}$. 
Indeed, if this is not the case then the fact that 
$x_1\ge \frac{\alpha}{i}$ implies that 
$x_2 <\frac{\alpha}{i}$; hence $x_j <\frac{\alpha}{i}$ for $j\ge 3$ and thus
$\sum_{j=1}^{i}x_i = (x_1+x_2) +  \sum_{j=3}^{i}x_j< \alpha$, a contradiction. 
Continuing in this manner we see that 
$\sum_{j=1}^{\ell}x_j\ge \ell\cdot \frac{\alpha}{i}$, for all $\ell\in\{1,\ldots,i\}$. 

We now claim that it holds $x_{i+1}\ge -\frac{\alpha}{n-i}$. 
Indeed, if this is not the case then $x_{i+1}< -\frac{\alpha}{n-i}$ and thus $\sum_{\ell=i+1}^n x_{\ell} < -\alpha$, a contradiction. Continuing in this manner we find that 
$\sum_{j=i+1}^{\ell}x_j\ge -\ell\cdot \frac{\alpha}{n-i}$, for all $\ell\in\{i+1,\ldots,n\}$. It follows that $P_{i,\alpha} \prec P$, as desired. 
\end{proof}

Now let $P\in B_{\alpha}^{(n)}$. Since, by Lemma~\ref{minimal}, it holds $P_{i,\alpha} \prec P$ for some $i\in [n-1]$,  it follows from Lemma~\ref{Schur1} that 
\begin{equation}\label{eqx4}
\Exp(|X_P|)\ge \Exp(|X_{P_{i,\alpha}}|) \, ,
\end{equation}
where $P_{i,\alpha}$ is the vector defined in~\eqref{p_vector},
and it is therefore enough to find a lower bound on $\Exp(|X_{P_{i,\alpha}}|)$. 
Towards this end, we first express $\Exp(|X_{P_{i,\alpha}}|)$ in terms of a hypergeometric random variable. 

Let $H_i\sim \text{Hyp}(n,i,k)$ be the number of positive elements in a sample without replacement of size $k$ from the population $P_{i,\alpha}$, and observe that 
\begin{equation}\label{x_ip}
X_{P_{i,\alpha}} = H_i\cdot\frac{\alpha}{i} - (k-H_i)\frac{\alpha}{n-i} = \frac{\alpha n}{i(n-i)} \left(H - \frac{ik}{n}\right) \, .
\end{equation}
Hence we have 
\begin{equation}\label{eqx5}
\Exp(|X_{P_{i,\alpha}}|) = \frac{\alpha n}{i(n-i)}\cdot \Exp\left(|H_i - \frac{ik}{n}|\right) \, .
\end{equation}
Summarizing the above, it follows from~\eqref{eqx3},~\eqref{eqx4} and~\eqref{eqx5} that for every $P\in B_{\alpha}^{(n)}$ it holds 
\[
\Prb(X_P> 0)\, \ge\, \frac{\Exp(|X_P|)}{2\alpha} \,\ge\, \frac{\frac{\alpha n}{i(n-i)}\cdot \Exp\left(|H_i - \frac{ik}{n}|\right)}{2\alpha} \,=\, \frac{n}{2i(n-i)}\cdot \Exp\left(|H_i - \frac{ik}{n}|\right) \, , 
\]
a bound which is \emph{independent} of $\alpha$. 
Therefore, since $\frac{k}{n}\cdot \sqrt{\frac{n-k}{n\cdot k}} = \sqrt{\frac{k}{n}}\cdot \frac{\sqrt{n-k}}{n}$,  the proof of Theorem~\ref{main_thm} will be complete  when the following result is proven.  

\begin{lemma}\label{mad_hyp} Let $H_i\sim\text{Hyp}(n,i,k)$, where $i\in [n-1]$ and $k\in [n-1]$, for some integer $n\ge 2$. Then it holds
\begin{equation}\label{eqx6}
\frac{n}{2i(n-i)}\cdot \Exp\left(|H_i - \frac{ik}{n}|\right) \,\ge\,\frac{e^{-1/3}}{8\sqrt{2\pi}} \cdot \sqrt{\frac{k}{n}}\cdot \frac{\sqrt{n-k}}{n} \, .
\end{equation}
\end{lemma}

\noindent
The remaining part of this section is devoted to the proof of Lemma~\ref{mad_hyp}. 
We break the proof in several lemmata.  
We begin with a result that provides a lower bound on the probability that a hypergeometric random variable takes the  value that is right above its mean. 

\begin{lemma}\label{integer_part}
Let $H\sim\text{Hyp}(n,i,k)$ be such that $\frac{ik}{n} \in [m-1,m]$, for some  $m \in \{2,\ldots, \min\{i,k\}-1\}$.  
Then
\[
\Prb(H = m) \ge \frac{e^{-1/3}}{16\sqrt{2\pi}} \cdot \sqrt{\frac{i(n-i)k(n-k)}{m(i-m)(k-m)(n-i-k+m) n}} \, .
\]
\end{lemma}
\begin{proof}
Observe  that 
\begin{equation}\label{m1}
 \Prb(H = m) =   \frac{\binom{i}{m}\cdot\binom{n-i}{k-m}}{\binom{n}{k}} = \frac{i! \cdot (n-i)!\cdot k!\cdot (n-k)!}{m!\cdot (i-m)!\cdot (k-m)!\cdot (n-i-k+m)!\cdot n!} \, .
\end{equation}
We distinguish three cases. 
Suppose first that the parameters $n,i,k,m$ satisfy $n-i-k+m\ge 2$.

We employ Robbins' version of Stirling's formula (see~\cite{Robbins}), which states that for all positive integers $n$ it holds 
\[
\sqrt{2\pi n} \left(\frac{n}{e}\right)^n e^{\frac{1}{12n+1}} < n! < \sqrt{2\pi n} \left(\frac{n}{e}\right)^n e^{\frac{1}{12n}} \, .
\]
This implies that the right hand side of~\eqref{m1} is larger than 
\begin{equation}\label{m2}
    A = \frac{e^{B}}{\sqrt{2\pi}}\cdot \sqrt{\frac{i(n-i)k(n-k)}{m(i-m)(k-m)(n-i-k+m) n}} \cdot C \, ,
\end{equation}
where 
\begin{eqnarray*}
B &=& \frac{1}{12i +1} + \frac{1}{12(n-i) +1} + \frac{1}{12k +1}+\frac{1}{12(n-k) +1} \\
&-& \frac{1}{12m} - \frac{1}{12(i-m)} - \frac{1}{12(k-m)} - \frac{1}{12(n-i-k+m)} - \frac{1}{12n} 
\end{eqnarray*}
and 
\[
C = \frac{i^i (n-i)^{n-i}k^k(n-k)^{n-k}}{m^m (i-m)^{i-m} (k-m)^{k-m}(n-i-k+m)^{n-i-k+m} n^n} \, .
\]
We first estimate $C$ from below, by showing that $C\ge \frac{1}{16}$. We then show that $B\ge -1/3$. To this end, 
note that we may write 
\begin{eqnarray*}
C &=& \frac{i^i (n-i)^{n-i}k^k(n-k)^{n-k}}{m^m (i-m)^{i-m} (k-m)^{k-m}(n-i-k+m)^{n-i-k+m} n^n} \\ 
&=&  \left(\frac{ik}{n m}\right)^m \cdot \left(\frac{k(n-i)}{n (k-m)}\right)^{k-m}
\cdot \left(\frac{i(n-k)}{(i-m)n}\right)^{i-m} \left(\frac{n-i}{n-i-k+m}\cdot\frac{n-k}{n}\right)^{n-i-k+m} \, .
\end{eqnarray*}
Since $\Exp(H)\in (m-1,m]$, it follows that $m-1<\frac{ik}{n}\le m$ and thus it holds 
\[
 \frac{ik}{n m} \ge  \frac{m-1}{m}
 \quad , \quad \frac{k(n-i)}{n(k-m)} \ge 1  \quad  \text{ and } \quad \frac{i(n-k)}{(i-m)n}\ge 1 \, .
\]
Moreover, the fact that $\frac{ik}{n}\ge m-1$ implies that  
\begin{equation}\label{m3}
\frac{n-i}{n-i-k+m} \ge \left(1 - \frac{1}{n-i-k+m}\right) \cdot \frac{n}{n-k} \,  ,
\end{equation}
and thus we conclude that 
\[
C \ge \left(1-\frac{1}{m}\right)^m \left(1 - \frac{1}{n-i-k+m}\right)^{n-i-k+m} \, .
\]
Since $m\ge 2$ and $n-i-k+m \ge 2$, by assumption,  the fact that the sequence $\{(1-\frac{1}{s})^s\}_{s\ge 1}$ is increasing implies   that  $C\ge \frac{1}{4}\cdot \frac{1}{4} = \frac{1}{16}$, as desired. 

Now we estimate $B$ from below. Since it clearly holds 
$\frac{1}{12i +1} - \frac{1}{12n}\ge 0$ and $\frac{1}{12 s}\le \frac{1}{12}$, for $s \in \{m,i-m,k-m,n-i-k+m\}$, we conclude that 
$B \ge -\frac{4}{12} = -\frac{1}{3}$. The result follows. 

Assume now that the parameters $n,i,k,m$ satisfy 
$n-i-k+m=0$. In this case it holds $n-i=k-m$ and thus we may write 
\[
\Prb(H=m) = \frac{\binom{i}{m}\binom{k-m}{k-m}}{\binom{i+k-m}{k}} 
\]
and $H\sim\text{Hyp}(i+k-m,i,k)$. 
Now it is easy to verify that it holds  $\frac{ik}{i+k-m} \ge m$. Since it also holds $\frac{ik}{i+k-m} \le m$, by assumption, we deduce that $i=m$. 
Hence  $\Prb(H=m)=1$. 

Assume now that the parameters $n,i,k,m$ satisfy 
$n-i-k+m=1$. In this case the proof is almost same as in the first case, and so we only sketch it.  Instead of~\eqref{m3}, we estimate  
\[
C \ge \frac{1}{4} \cdot \frac{(n-i)(n-k)}{n}= \frac{1}{4} \cdot \frac{(k-m+1)(i-m+1)}{k+i-m+1} \ge \frac{1}{4} \, .
\] 
In all cases, it holds  
\[
\Prb(H=m) \ge \frac{e^{-1/3}}{4\sqrt{2\pi}} \cdot \sqrt{\frac{i(n-i)k(n-k)}{m(i-m)(k-m)(n-i-k+m)n}} \, ,
\]
as desired. 
\end{proof}

The next result establishes~\eqref{eqx6} when $\frac{ik}{n}$ is neither too small nor too large. 

\begin{lemma}\label{mad_lb}
Let $H\sim\text{Hyp}(n,i,k)$ be such that $\frac{ik}{n} \in [m-1,m]$, for some $m \in \{2,\ldots, \min\{i,k\}-1\}$.  
Then 
\[
\frac{n}{2i(n-i)}\cdot\Exp(|H-ik/n|) \ge \frac{e^{-1/3}}{8\sqrt{2\pi}} \cdot \sqrt{\frac{k}{n}}\cdot \frac{\sqrt{n-k}}{n} \, .
\]  
\end{lemma}
\begin{proof}
We  use the following expression for the mean absolute deviation of a hypergeometric random variable (see~\cite[Formula~(3.3), p.554]{ramasubban1958mean}):
\[
\Exp(|H-ik/n|) = \frac{2m}{n}\cdot (n-i-k+m)\cdot\Prb(H=m) \, .
\]
Hence we may write  
\[
A_i := \frac{n}{2i(n-i)}\cdot\Exp(|H-ik/n|) = \frac{m}{i}\cdot \frac{n-i-k+m}{n-i}\cdot \Prb(H=m)
\]
and Lemma~\ref{integer_part} yields
\begin{eqnarray*}
A_i &\ge& \frac{e^{-1/3}}{16\sqrt{2\pi}} \cdot  \frac{m}{i}\cdot \frac{n-i-k+m}{n-i}\cdot  \sqrt{\frac{i(n-i)k(n-k)}{m(i-m)(k-m)(n-i-k+m) n}}\\
&=& \frac{e^{-1/3}}{16\sqrt{2\pi}} \cdot  \sqrt{\frac{m(n-i-k+m) k (n-k)}{i(n-i)(i-m)(k-m)n}} \, .
\end{eqnarray*}
Now observe that $\frac{n-i-k+m}{n-i} \ge \frac{n-i-k+\frac{ik}{n}}{n-i} = \frac{n-k}{n}$, which in turn implies that 
\[
A_i \ge \frac{e^{-1/3}}{16\sqrt{2\pi}} \cdot  \frac{n-k}{n} \cdot\sqrt{\frac{m}{i}}\cdot\sqrt{\frac{k}{(i-m)(k-m)}} .
\]
Moreover, since $m\ge\frac{ik}{n}$, it holds $\frac{m}{i} \ge\frac{k}{n}$ and 
\[
(i-m)\cdot (k-m) \le (i-ik/n)\cdot (k-ik/n) = \frac{i k (n-i)(n-k)}{n^2} \, .
\]
Hence 
\[
A_i \ge \frac{e^{-1/3}}{16\sqrt{2\pi}} \cdot  \frac{n-k}{n} \cdot\sqrt{\frac{k}{n}}\cdot \sqrt{\frac{n^2}{i(n-i)(n-k)}} \ge \frac{e^{-1/3}}{16\sqrt{2\pi}} \cdot  \frac{n-k}{n} \cdot \sqrt{\frac{k}{n}}\cdot\sqrt{\frac{4}{n-k}} \, ,
\]
where the last inequality follows from the fact that $i(n-i)\le n^2/4$. 
The result follows. 
\end{proof}

The remaining cases, i.e., when $\frac{ik}{n}$ is either too small or too large, are considered in the next result, which finishes the proof of Lemma~\ref{mad_hyp} and the proof  of Theorem~\ref{main_thm}.

\begin{lemma}\label{mad_lb1}
Let $H_i\sim\text{Hyp}(n,i,k)$ be such that $\frac{ik}{n} \in (0,1] \cup (\min\{i,k\}-1, \min\{i,k\}]$.   
Then 
\[
\frac{n}{2i(n-i)}\cdot\Exp(|H_i-ik/n|)\, \ge \, \frac{e^{-1/3}}{8\sqrt{2\pi}} \cdot \sqrt{\frac{k}{n}}\cdot \frac{\sqrt{n-k}}{n} \, .
\]    
\end{lemma}
\begin{proof}
Suppose first that $\frac{ik}{n}\in (0,1]$ and 
observe that it holds  
\begin{eqnarray*}
\Exp(|H_i-ik/n|) &=& \frac{ik}{n}\cdot \Prb(H_i=0) + \left(1-\frac{ik}{n}\right) \cdot \Prb(H_i=1) +  \sum_{j\ge 2}\left(j-\frac{ik}{n}\right)\cdot \Prb(H_i=j) \\
&\ge&  \frac{ik}{n}\cdot \Prb(H_i=0) + \left(1-\frac{ik}{n}\right) \cdot \Prb(H_i\ge 1) \, .
\end{eqnarray*}
We first consider the case $\frac{ik}{n}<1$. 
In this case, we will show a bit more. Namely, we show that 
\begin{equation}\label{eqmad1}
\frac{n}{2i(n-i)}\cdot\Exp(|H_i-ik/n|)\ge \frac{1}{2}\cdot \frac{k}{n} \, ,\, \text{ when } \, \frac{ik}{n}<1 \, .
\end{equation}
The result then follows from the fact that $\frac{k}{n}\ge \sqrt{\frac{k}{n}}\cdot \frac{\sqrt{n-k}}{n}$. 
To this end, note that we may write 
\begin{eqnarray*}
A_i &:=& \frac{n}{2i(n-i)}\cdot\Exp(|H_i-ik/n|) \\ 
&\ge& \frac{n}{2i(n-i)} \cdot \left(\frac{ik}{n}\cdot \Prb(H_i=0) + \left(1-\frac{ik}{n}\right) \cdot \Prb(H_i\ge 1) \right) \, .
\end{eqnarray*}
We distinguish two cases. Assume first that $\frac{ik}{n}\le \frac{1}{2}$. Then $\frac{ik}{n} \le 1-\frac{ik}{n}$ and thus 
\[
A_i \ge \frac{n}{2i(n-i)} \cdot \frac{ik}{n}\cdot \Prb(H_i\ge 0) = \frac{k}{2(n-i)}\ge \frac{1}{2}\cdot \frac{k}{n} \, , 
\]
which proves~\eqref{eqmad1}. 
Now assume that $\frac{1}{2}< \frac{ik}{n} <  1$. Then, we have 
\[
A_i \ge \frac{n}{2i(n-i)}\cdot \frac{n-ik}{n}\cdot \Prb(H_i\ge 0) = 
\frac{n-ik}{2i(n-i)} \, .
\]
Now observe that, since $i\ge 1$, it holds  $2\frac{ik}{n} \le 1 + \frac{i^2k}{n}$; this implies that   $\frac{n-ik}{n-i}\ge \frac{ik}{n}$ and therefore $A_i\ge \frac{1}{2}\cdot \frac{k}{n}$, thus establishing~\eqref{eqmad1}. 

Now suppose that $\Exp(H_i)=\frac{ik}{n} =1$. 
We may assume that $i>1$; if $i=1$, then $k=n$ and $A_1 > \frac{1}{2}\cdot \frac{k}{n}$. 
Now observe that, since $i>1$, we may write 
\begin{eqnarray*}
A_i &:=&  \frac{ik}{2i(ik-i)}\cdot\left( \Prb(H_i=0) + \sum_{j\ge 2} (j-1)\cdot\Prb(H_i=j) \right) \\
&=& \frac{k}{2i(k-1)}\cdot\big( \Prb(H_i=0) + \Exp(H_i) -\Prb(H_i=1) -\Prb(H_i\ge 2) \big) \\
&=& \frac{k}{i(k-1)}\cdot  \Prb(H_i=0) 
\,=\, \frac{k}{i(k-1)}\cdot \frac{\binom{ik-i}{k}}{\binom{ik}{k}}
\,=\, \frac{ik-i-k+1}{i^2(k-1)} \cdot \Prb(H_i=1) \, .
\end{eqnarray*}
Since the median of $H_1$ is equal to its mean  (see~\cite{siegel2001median}), it follows that 
$\Prb(H_i \le 1)\ge \frac{1}{2}$, and so we either have  
$\Prb(H_i =0)\ge \frac{1}{4}$ or $\Prb(H_i =1)\ge \frac{1}{4}$.
If $\Prb(H_i =0)\ge \frac{1}{4}$, then 
\[
A_i = \frac{k}{i(k-1)}\cdot \Prb(H_i=0) \ge \frac{1}{4i}=\frac{1}{4}\cdot \frac{k}{n}\, . 
\]

If  $\Prb(H_i=1)\ge \frac{1}{4}$ then
\[
A_i = \frac{ik-i-k+1}{i^2(k-1)} \cdot \Prb(H_i=1)\ge \frac{1}{4}\cdot\frac{ik-i-k+1}{i^2(k-1)} = \frac{1}{4}\cdot\frac{i-1}{i}\cdot\frac{1}{i} \ge \frac{1}{8}\cdot\frac{1}{i}= \frac{1}{8}\cdot \frac{k}{n}\, , 
\]
since $i>1$. The result follows for the case $\frac{ik}{n}\in (0,1]$.

Now we consider the case $\frac{ik}{n} \in  (m-1, m]$, where $m=\min\{k,i\}$. Let $W_i = k - H_i$ and observe that $W_i\sim\text{Hyp}(n,n-i,k)$ and that 
\[
\Exp(|H_i - ik/n|) = \Exp(|W_i - (n-i)k/n|) \, .
\]
Since $m-1 < \frac{ik}{n} \le m$ it follows that $k-m \le \frac{(n-i)k}{n} < k-m+1$. If $k-m \ge 1$, then the result follows from Lemma~\ref{mad_lb} and the fact that $\Exp(|H_i - ik/n|) = \Exp(|W_i - (n-i)k/n|)$. 
If $k-m=0$, then the result follows from the case $\frac{ik}{n}\in (0,1]$ that has already been considered in the beginning of the proof. The result follows.  
\end{proof}

The proof of Theorem~\ref{main_thm} is complete. We conclude this section with a proof of Corollary~\ref{main_thm4}. 

\begin{proof}[Proof of Corollary~\ref{main_thm4}]
Consider the population $-P= (-x_1,\ldots,-x_n)$, and note that $-P\in B_{\alpha}^{(n)}$. We know from Theorem~\ref{main_thm} that $\Prb(X_{-P}>0) \ge \frac{e^{-1/3}}{8\sqrt{2\pi}} \cdot \frac{k}{n}\cdot \sqrt{\frac{n-k}{n\cdot k}}$. 
The result follows upon observing that  
\[
\Prb(X_P\ge 0) = \Prb(X_{-P}\le 0) = 1- \Prb(X_{-P}>0) \, .
\]
\end{proof}

\section{Proof of Theorem~\ref{main_thm2}}
\label{sec:2}

Let $\mathbf{1}_A$ denote the indicator of the event $A$. 
We begin with the following, well-known, result from the theory of majorization. 

\begin{lemma}\label{cvx_lemma}
Let $P_1,P_2\in\mathbb{R}^n$ be two populations such that $P_1\prec P_2$, and let $f$ be a continuous convex function. 
Then it holds 
\[
\Exp(f(X_{P_1})) \le \Exp(f(X_{P_2})) \, .
\]
\end{lemma}
\begin{proof}
This is a well-known result due to Karlin (see~\cite[p.~455]{Marshall_Olkin_Arnold}). 
\end{proof}

\begin{lemma}\label{y_lem}
Let $Y$ be a random variable such that $Y\in [-1,1]$, and let 
$y\in (0,1)$ be fixed. Let $f:[-1,1]\to [-1,1]$ be the function defined by 
\begin{equation*}
f(x)=\begin{cases}
\frac{1}{1-y}\cdot x + \frac{y}{1-y}, &\text{ when } x\in [-1\, ,\, 1-2y]\\
1, &\text{ when } x \in (1-2y\, ,\, 1].
\end{cases}
\end{equation*}
Then 
\[
\Prb(Y> -y) \ge \Exp(f(Y)) \, .
\]
\end{lemma}
\begin{proof}
Observe that the function $f$ is such that 
\[
\mathbf{1}_{\{Y> -y\}} \ge f(Y) \, .
\]
The result follows upon taking expectations. 
\end{proof}

\begin{proof}[Proof of Theorem~\ref{main_thm2}]
Consider the population $-P = (-x_1,\ldots,-x_n)$, and note that $-P\in B_{\alpha}^{(n)}$. Moreover, since $X_{-P} \sim -X_P$, we have  
\[
\Prb(X_P \ge t) = \Prb(-X_P \le -t) =\Prb(X_{-P} \le -t) = 1 - \Prb(X_{-P} > -t) \, ,
\]
and it is therefore enough to show that for any population, let us denote it again  
$P=(x_1,\ldots,x_n)$, which satisfies $\sum_i x_i=0$ and 
$\sum_i |x_i|=2\alpha$, for some $\alpha>0$, it holds 
\[
\Prb(X_P >- t) \ge \min\left\{1\, ,\, \frac{t}{\alpha -t}\right\}\cdot\left(1 - \frac{2k(n-k)}{n(n-1)}\right) \, .
\]
Set $Y_P = \frac{1}{\alpha}\cdot X_P$ and observe that $Y_P\in [-1,1]$. Hence Lemma~\ref{y_lem} yields 
\[
\Prb(X_P > -t) = \Prb\left(Y_P > -\frac{t}{\alpha}\right) \ge \Exp(f(Y_P)) \, ,
\]
where $f:[-1,1]\to [-1,1]$ is the function 
$f(x)=\begin{cases}
\frac{\alpha}{\alpha-t}\cdot x + \frac{t}{\alpha-t}, &\text{ when }\, x\in [-1\, ,\, 1-\frac{2t}{\alpha}]\\
1, &\text{ when }\, x \in (1-\frac{2t}{\alpha} \, ,\, 1]
\end{cases}$. 

Now observe that $P\prec P^*$, where $P^* = (\alpha,0,\ldots,0,-\alpha)\in B_{\alpha}^{(n)}$. 
Since the function $f$ is concave, it follows from Lemma~\ref{cvx_lemma} 
that $\Exp(f(Y_P)) \ge \Exp(f(Y_{P^*}))$. 
Hence it holds 
\[
\Prb(X_P > -t) \ge \Exp(f(Y_{P^*})) \, ,
\]
and we are left with calculating $\Exp(f(Y_{P^*}))$.
Now observe that $Y_{P^*} \in \{-1,0,1\}$ and that 
\[
\Prb(Y_{P^*}=1)=\Prb(Y_{P^*}=-1)= \frac{k(n-k)}{n(n-1)} \quad  \text{and} \quad
\Prb(Y_{P^*}=0) = 1 - \frac{2k(n-k)}{n(n-1)} \, 
\]

We distinguish two cases. 
Suppose first that $\frac{t}{\alpha-t} \ge 1$, so that $t\in [\alpha/2, \alpha)$. Observe that in this case it holds $f(0) =1\, , \, f(1)=1$ and $f(-1)=-1$; thus   
\begin{eqnarray*}
\Exp(f(Y_{P^*})) &=& f(-1)\cdot   \frac{k(n-k)}{n(n-1)} + f(0)\cdot  \left( 1 - \frac{2k(n-k)}{n(n-1)}\right) + f(1)\cdot  \frac{k(n-k)}{n(n-1)} \\ 
&=& 1 - \frac{2k(n-k)}{n(n-1)} \, .
\end{eqnarray*}

Now suppose that $\frac{t}{\alpha-t} < 1$, so that $t\in (0, \alpha/2)$. Then $f(0) =\frac{t}{\alpha-t}\, , \, f(1)=1$ and $f(-1)=-1$; hence   
\[
\Exp(f(Y_{P^*})) = \frac{t}{\alpha-t} \cdot  \left( 1 - \frac{2k(n-k)}{n(n-1)}\right) \, ,
\]
and the result follows. 
\end{proof}

\section{Concluding remarks}
\label{sec:3}

In this article, we report bounds on the tail probability \(\mathbb{P}(X_{P} \geq t)\) in the context of sampling without replacement, using basic tools from the theory of majorization. When \(t = 0\), we obtain a lower bound  which is of the form \(c \cdot \frac{k}{n} \cdot \sqrt{\frac{n-k}{nk}}\),  for some constant \(c<1\). Our bound is, of course, smaller than the bound \(\frac{k}{n}\) which is speculated by the MMS conjecture but has the advantage that it does not require any additional constraints on \(n\) and \(k\), thus  making it more widely applicable. Additionally, we establish a corresponding upper bound.  

When \(t > 0\), we obtain an upper bound that incorporates information on the absolute deviation of the population. 
To demonstrate its usefulness in scenarios where the absolute deviation is the only accessible information, we compare our bound with worst-case instances of certain refined versions of Theorem~\ref{thm_Hoeffding}, which have been reported  in~\cite{Bardenet_Maillard}. 
To ensure a fair comparison of the bounds, we express them in terms of sample averages. If we were to compare the bounds using a fixed threshold $t = k \varepsilon$, the value of \( \varepsilon \) would vary as \( k \) ranges from \( 1 \) to \( n \), making direct comparisons inconsistent. Hence, we present the results using sample averages: given a population \( P = (x_1, \ldots, x_n) \) and a sample size \( k \in [n-1] \), we select a subset \( \mathbb{I} \in \binom{[n]}{k} \) uniformly at random and define the sample average as \( A_P = \frac{1}{k} \sum_{i \in \mathbb{I}} x_i = \frac{1}{k} X_P \).  The bound to be compared are summarized in the following.

\begin{theorem}
\label{thms}
Let $P=(x_1,\ldots,x_n)$ be a population of size $n\ge 2$ such that $\sum_{i=1}^{n} x_i=0$, and let $k\in [n-1]$. 
\begin{enumerate}
\item (Bardenet--Maillard~\cite[Theorem~2.4]{Bardenet_Maillard}) \,  If\, $a = \min_{i\in [n]} x_i$\, and\, $b = \max_{i\in [n]} x_i$, then 
\begin{equation*}\label{Hoeff_bound_var}
\Prb(A_P \ge\varepsilon) \le \exp\left\{-\frac{2k\varepsilon^2}{(1-\frac{k}{n})(1+\frac{1}{k})(b-a)^2}\right\} \, , \, \text{ for } \, \varepsilon >0 \, .
\end{equation*}  
\item (Bardenet--Maillard~\cite[Theorem~3.5]{Bardenet_Maillard}) \, If\, $a = \min_{i\in [n]} x_i$ ,  $b = \max_{i\in [n]} x_i$, $\sigma^2 = \frac{1}{n}\sum_{i=1}^{n}x_{i}^2$, then for any $\varepsilon>0$ and $\delta\in [0,1]$, it holds 
\[
\Prb(A_P \ge\varepsilon) \le \exp\left\{-\frac{k\varepsilon^2/2}{\gamma^2 + (2/3)(b-a)\varepsilon}\right\}+\delta,
\]
where 
\begin{equation*}
\begin{aligned}
\gamma^2 = (1-f_k)\left(\frac{k+1}{k}\sigma^2+\frac{n-k-1}{k}c_{n-k-1(\delta)}\right),\\
c_{n-k-1}(\delta)=\sigma(b-a)\sqrt{\frac{2\log(1/\delta)}{n-k-1}}\,\, ,\quad f_{k}=\frac{k}{n}.
\end{aligned}    
\end{equation*}

\item (Theorem~\ref{main_thm2}, restated) \, If\, $\sum_{i=1}^{n}|x_i|=2\alpha$, for some $\alpha>0$, then 
\[
\Prb(A_P\ge \varepsilon)\, \le\, 1 - \min\left\{1\, ,\, \frac{\varepsilon k}{\alpha -\varepsilon k}\right\}\cdot\left(1 - \frac{2k(n-k)}{n(n-1)}\right) \, , \, \text{ for } \, \varepsilon\in (0,\alpha/k) \, .
\]
\end{enumerate}
\end{theorem}

As mentioned earlier, when the mean and the absolute deviation of the population, i.e., the parameter \( 2\alpha \), are the only two pieces of information available, the first two bounds in Theorem~\ref{thms} may assume worst-case values for the range \( b - a \) and the variance \( \sigma^2 \); thus we suppose that \(b-a = 2\alpha \) and \(\sigma^2= 2\alpha^2 \) in the first two bounds of Theorem~\ref{thms}. In our comparisons, we set \( n = 100 \) and \( \alpha = 1 \). Figure~\ref{fig:1} depicts the three bounds for \( \varepsilon = 0.001, 0.005, 0.01 \), and different values of the sample size $k$. As is shown in the figure, Theorem~\ref{main_thm2} provides a significantly tighter bound in this setting, thus demonstrating its advantage in instances where one has only access to the mean and the  absolute deviation of the population.

\begin{figure}[htbp]
    \centering
    \begin{subfigure}[b]{0.65\textwidth}
        \centering
        \includegraphics[width=\textwidth]{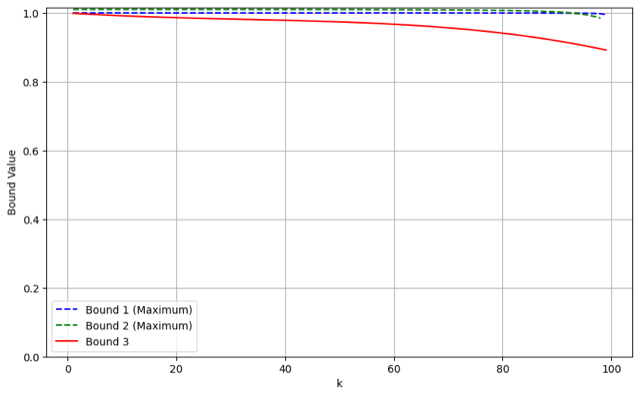} 
        \caption{$\varepsilon = 0.001$}
        \label{fig:1a}
    \end{subfigure}
    \hfill
    \begin{subfigure}[b]{0.65\textwidth}
        \centering
        \includegraphics[width=\textwidth]{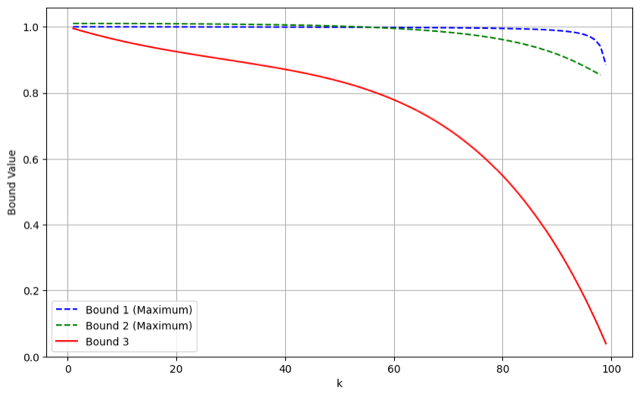} 
        \caption{$\varepsilon = 0.005$}
        \label{fig:1b}
    \end{subfigure}
    \hfill
    \begin{subfigure}[b]{0.65\textwidth}
        \centering
        \includegraphics[width=\textwidth]{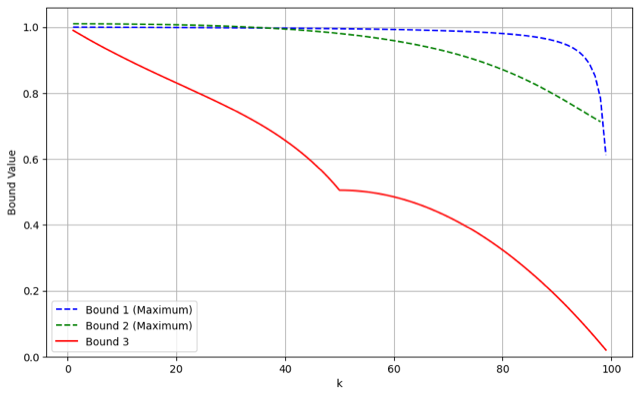} 
        \caption{$\varepsilon = 0.01$}
        \label{fig:1c}
    \end{subfigure}
    \caption{Comparison of the bounds for different values of $\varepsilon$}
    \label{fig:1}
\end{figure}
\newpage

\appendix
\section{A lower bound}
\label{appendix1}
We present a lower bound on $\Prb(X_{P}\geq t)$, for $t>0$, which applies to moderate values of the parameter $t>0$. 

\begin{proposition}\label{main_thm3}
Let $P=(x_1,\ldots,x_n)$ be a population of size $n\ge 2$ such that $\sum_{i=1}^{n} x_i=0$ and $\sum_{i=1}^{n} |x_i|=2\alpha$, for some $\alpha>0$.  
Let $X_P$ denote the sum of $k\in [n-1]$ elements that are sampled  without replacement from $P$, and let $t\in \left(0\, ,\,\frac{4k(n-k)}{n^2(n-1)}\cdot \alpha\right)$ be fixed. Then   
\[
\Prb(X_P\ge t)\, \ge\, \frac{2\alpha}{(\alpha-t)} \cdot \frac{k(n-k)}{n^2(n-1)} -  \frac{t}{2(\alpha-t)}  \, .
\]
\end{proposition}

\begin{lemma}\label{y_lem2}
Let $Y$ be a random variable such that $Y\in [-1,1]$, and let 
$y\in [0,1)$ be fixed. Let $f:[-1,1]\to [-1,1]$ be the function defined by 
\begin{equation*}
f(x)= \frac{1}{2(1-y)}\cdot x^2 + \frac{1}{2}\cdot x - \frac{y}{2(1-y)} \, .
\end{equation*}
Then 
\[
\Prb(Y\ge y) \ge \Exp(f(Y)) \, .
\]
\end{lemma}
\begin{proof}
Observe that the function $f$ is convex and satisfies $f(-1)=f(y)=0$ and $f(1)=1$. Hence it holds 
\[
\mathbf{1}_{\{Y\ge y\}} \ge f(Y) \, ,
\]
and the result follows upon taking expectations. 
\end{proof}

\begin{proof}[Proof of Proposition~\ref{main_thm3}]
Suppose that $P\in B_{\alpha}^{(n)}$, for some $\alpha>0$. 
Set $Y_P = \frac{1}{\alpha}\cdot X_P$ and observe that $Y_P\in [-1,1]$. Hence Lemma~\ref{y_lem2} yields 
\[
\Prb(X_P \ge -t) = \Prb\left(Y_P \ge -\frac{t}{\alpha}\right) \ge \Exp(f(Y_P)) \, ,
\]
where $f:[-1,1]\to [-1,1]$ with  
$f(x)= \frac{\alpha}{2(\alpha-t)}\cdot x^2 + \frac{1}{2}\cdot x - \frac{t}{2(\alpha-t)}$. 
Now Lemma~\ref{minimal} implies that $P_{i,\alpha}\prec P$, where $P_{i,\alpha}\in B_{\alpha}^{(n)}$ is the vector defined in~\eqref{p_vector}. 
Since the function $f$ is convex, it follows from Lemma~\ref{cvx_lemma} 
that $\Exp(f(Y_P)) \ge \Exp(f(Y_{P_{i,\alpha}}))$. 
Hence, we have  
\[
\Prb(X_P \ge t) \ge \Exp(f(Y_{P_{i,\alpha}})) \, . 
\]
Now we calculate 
\[
\Exp(f(Y_{P_{i,\alpha}})) = \frac{\alpha}{2(\alpha-t)}\cdot\frac{1}{\alpha^2} \Exp(X_{P_{i,\alpha}}^{2})  - \frac{t}{2(\alpha-t)} \, .
\]
From~\eqref{x_ip} we know that  $X_{P_{i,\alpha}}  = \frac{\alpha n}{i(n-i)} \left(H_{i} - \frac{ik}{n}\right)$, where 
$H_i\sim\text{Hyp}(n,i,k)$. Hence 
\[
\Exp(X_{P_{i,\alpha}}^{2}) = \left(\frac{\alpha n}{i(n-i)}\right)^2\cdot \text{Var}(H_i) =  \left(\frac{\alpha n}{i(n-i)}\right)^2\cdot k\cdot \frac{i}{n}\cdot \frac{n-i}{n}\cdot \frac{n-k}{n-1} \, , 
\]
and thus 
\[
\Exp(f(Y_{P_{i,\alpha}})) = \frac{\alpha}{2(\alpha-t)} \cdot \frac{k(n-k)}{i(n-i)(n-1)} -  \frac{t}{2(\alpha-t)} 
\ge  \frac{2\alpha}{(\alpha-t)} \cdot \frac{k(n-k)}{n^2(n-1)} -  \frac{t}{2(\alpha-t)} \, ,
\]
where the inequality follows from the fact that $i(n-i)\le n^2/4$. 
\end{proof}

\end{document}